\documentclass[12pt]{article}

 \usepackage{amsmath,amssymb,amscd,amsthm,esint}

\usepackage{graphics,amsmath,amssymb,amsthm,mathrsfs,amsfonts}

\usepackage{sidecap}
\usepackage{float}
\usepackage{extarrows}
\usepackage{booktabs}
\usepackage{verbatim}
\usepackage{hyperref}
\usepackage[usenames,dvipsnames]{xcolor}

\usepackage{titlesec}
\usepackage{titletoc}

\usepackage[titletoc]{appendix}
\usepackage{enumerate}

\newtheorem{thm}{Theorem}[section]
\newtheorem{lemma}[thm]{Lemma}

\newtheorem{prop}[thm]{Proposition}

\newtheorem{remark}[thm]{Remark}

\newcommand{\varep}{\varepsilon}

\def\XXint#1#2#3{{\setbox0=\hbox{$#1{#2#3}{\int}$}
         \vcenter{\hbox{$#2#3$}}\kern-.5\wd0}}

\def\R{\mathbb{R}}

\def\e{\varepsilon}

\numberwithin{equation}{section}

\begin{document}

\title{$W^{1,p}$ estimates for Schr\"odinger equation in the region above a convex graph}

\author{Ziyi Xu}
\date{}
\maketitle

\begin{abstract}
  We investigate the $W^{1,p}$ estimates of the Neumann problem for the Schr\"odinger equation $-\Delta u+ V u={\rm div}(f)$ in the region above a convex graph. For any $p>2$, we obtain a sufficient condition for the $W^{1,p}$ solvability. As a result, we obtain sharp $W^{1,p}$ estimate
$$\|\nabla u\|_{L^p(\Omega)}+\|V^\frac{1}{2}u\|_{L^p(\Omega)}\leq C\|f\|_{L^p(\Omega)}$$
  for $1 <p<\infty$ with $d\geq2$ under the assumption that $V$ is a $B_\infty$ weight.
\end{abstract}


\section{\bf Introduction}\label{section-1}

The purpose of this paper is to establish $W^{1,p}$ solvability for Schr\"odinger operator in the region above a convex graph. Precisely, let
$$\Omega=\{(x',t): x'\in \mathbb{R}^{d-1}, t\in\mathbb{R} \text{ and }t>\phi(x')\}\subset\mathbb{R}^{d},$$
where $\phi:\mathbb{R}^{d-1}\to\mathbb{R}$ is a convex function with $\|\nabla\phi\|_{L^\infty(\mathbb{R}^{d-1})}\leq M$.
For $f\in L^p(\Omega,\R^d)$ and $g\in B^{-\frac 1p,p}(\partial\Omega)$, we consider the Schr\"odinger equation
  \begin{equation}\label{NP}
  \begin{aligned}
    \begin{cases}
      -\Delta  u+V u={\rm div}\,f \quad\ \,\text{in } \Omega,\\[0.1cm]
      \displaystyle\frac{\partial u}{\partial n}  =-f\cdot n+g \quad  \text{on } \partial\Omega,\\[0.2cm]
      u\in W^{1,p}(\Omega)
    \end{cases}
  \end{aligned}
  \end{equation}
Following the tradition and physical significance, $V$ is referred to be the electric potential. Throughout this paper, we assume that $0<V\in B_\infty$, i.e., $V\in L_{loc}^\infty(\mathbb{R}^d)$, and there exists a constant $C>0$ such that, for all ball $B\subset \mathbb{R}^d$
\begin{equation}\label{V}
    \|V\|_{L^\infty(B)}\leq  C\fint_B  V\,dx.
\end{equation}
Examples of $B_\infty$ weights are $|x|^a$ with $0\leq a<\infty$.

To state the main result of the paper, let $n$ denote the outward unit normal to $\partial\Omega$ and $B^{\alpha,p}(\partial\Omega)$ with $0<\alpha<1$ and $1<p<\infty$ denote the Besov spaces on $\partial\Omega$.

\begin{thm}\label{thm1}
  Let $\Omega$ be the region above a convex graph. Suppose $V(x)>0$ a.e. satisfies \eqref{V}. Then given $f\in L^p(\Omega,\R^d)$ and $g\in B^{-\frac 1p,p}(\partial\Omega)$ with
  $$1<p<\infty,$$
  the Neumann problem \eqref{NP} is uniquely solvable. Moreover, the solution satisfies
  \begin{equation}
    \label{result-NP}
    \|\nabla u\|_{L^p(\Omega)}+\|V^{\frac{1}{2}} u\|_{L^p(\Omega)}\leq C\left\{\|f\|_{L^p(\Omega)}+\|g\|_{B^{-\frac 1p,p}(\partial\Omega)}\right\},
  \end{equation}
  where $C$ depending only on $d,p$ and the Lipschitz character of $\Omega$.
\end{thm}

\begin{remark}
  The range of $p$ is sharp even for the Laplacian.
\end{remark}

The $W^{1,p}$ estimates for inhomogeneous equation $\Delta u=F$ in bounded non-smooth domains have been studied extensively.
Indeed, it has been known since the 1990s that the Dirichlet problem and the Neumann problem is $W^{1,p}$ solvable in bounded Lipschitz domains for $\frac32-\e<p<3+\e$ when $d\geq3$ where $\varep>0$ depends on
$\Omega$.  For a general second order elliptic equation with coefficients belonging to $VMO$, $W^{1,p}$ estimates are reduced to the weak reverse H\"older estimates for local solutions and are established for $\frac32-\e<p<3+\e$ when $d\geq3$. And the ranges of $p$ are sharp (see \cite{Fabes-1998,JK-1995, Geng-2012,Shen-2005}; also see \cite{AQ-2002,Byun-2005,Byun-Wang-2004,Byun-Wang-2005,Byun-Wang-2008,Dong-Kim-2010} for
references on related work on boundary value problems in bounded Lipschitz domains or  Reifenberg flat domains). It is worth noting that for every $p>3$ and $d\geq3$ there is a Lipschitz domain such that $\nabla u$ cannot belong to $L^p(\Omega)$ even if the right side $F$ is in $C^\infty$.
If a slightly stronger smoothness condition is imposed, that $\Omega$ is a bounded convex domain, the $W^{1,p}$ solvability was
essentially established in \cite{Geng-2018} for the sharp range $1<p<\infty$. Regarding the convexity of $\Omega$, with the analysis tools developed in \cite{GS-JFA-2010} at disposal, the weak reverse H\"older inequality
\begin{equation}\label{rh}
  \left(\fint_{B(x,r)\cap\Omega}|\nabla u|^pdx\right)^{\frac1p}\leq C\left(\fint_{B(x,2r)\cap\Omega}|\nabla u|^2dx\right)^{\frac12},\quad\text{for } p>3
\end{equation}
which is the sufficient condition to the $W^{1,p}$ estimate is established.

For Schr\"odinger equations, Z. Shen \cite{Shen-2024} obtained the $W^{1,p}$ estimate for Dirichlet problems for $2 < p < 3 + \e$ when $d\geq 3$, and $2 < p < 4+\e$ when $d = 2$ if $\Omega$ is a bounded Lipschitz domain and for $2 < p <\infty$ if $\Omega$ is a bounded $C^1$ domain under the assumption that the potential $V$ is positive and bounded.
In the region above a Lipschitz graph, Z. Shen \cite{Shen-1994-IUMJ} established the $L^{p}$ solvability for the Neumann problem and the Dirichlet problem if $V\in B_\infty$. The $W^{1,p}$ solvability is formulated in forthcoming paper \cite{geng-xu-submitted}. One may notice that the region above a Lipschitz graph maybe unbounded, and results in bounded domains do not work. And it worthwhile to flagged up that, in \cite{Shen-1994-IUMJ}, by the decay of solutions at infinity and the limit $R\to\infty$, the results in $\Omega_R=\{(x',t): |x'|<R \text{ and } \phi(x')<t<\phi(x')+R\}$ can be extended to $\Omega$. We remark that in \cite{Shen-1995}, Z. Shen established $\|\nabla u\|_{L^p(\mathbb{R}^d)}+ \|V^{\frac12} u\|_{L^p(\mathbb{R}^d)}\leq C\|f\|_{L^p(\mathbb{R}^d)}$ where $1\leq p\leq2q$ and $q\geq d$ for $-\Delta u+Vu={\rm div} f$ in $\mathbb{R}^d$ with $V\in B_q$.  
More related work about the Schr\"odinger operator refers to \cite{auscher-2007,lee-ok-2020,MT-1999-JFA,Shen-1999}.

Motivated by \cite{Geng-2018,Shen-1994-IUMJ}, we extend the results to the Schr\"odinger operator $-\Delta+ V $ in the region above a convex graph.
Our proof to Theorem \ref{thm1} follows the approach in \cite{Geng-2012}. Employing a real-variable perturbation argument and John-Nirenberg inequality, we give a sufficient condition for $W^{1,p}$ estimates for weak solutions of \eqref{NP} with $g = 0$. Roughly speaking,
we prove that for $p > 2$, if the reverse H\"older's inequality
\begin{equation}\label{rh2}
  \bigg\{\fint_{B(x,r)\cap\Omega}(|\nabla v|+V^{\frac12} |v|)^p\bigg\}^{\frac{1}{p}}\leq C\bigg\{\fint_{B(x,2r)\cap\Omega}(|\nabla v|+V^{\frac12} |v|)^2\bigg\}^{\frac{1}{2}}
\end{equation}
holds for the solutions of $-\Delta v+Vv=0$ in $\Omega$ and $\frac{\partial v}{\partial n}=0$ on $B(x,2r)\cap\Omega$, then the $W^{1,p}$ estimate is established. 
Following similar line of \cite{GS-JFA-2010}, we demonstrate the condition \eqref{rh2} by the improved Fefferman-Phong inequality, the estimates for the Fefferman-Phong-Shen maximal function $m(x, V)$ as well as the convexity of $\Omega$. For the case of $1<p<2$, duality arguments and the estimates for the Neumann functions
$$\int_{\Omega} |\nabla_y N(x, y)|\,dy\leq Cm(x,V)^{-1}$$ and
$$\int_{\Omega} |\nabla_x N(x, y)|m(y,V)^q\,dy\leq Cm(x,V)^{q-1}\quad\text{for integer } q\geq1$$
play significant roles.

The present paper can be split into three portions. In the first part, we collect some known results for
the Fefferman-Phong-Shen maximal function, boundary $L^\infty$ estimate and estimate for the Neumann function. The second portion presents a sufficient condition of $W^{1,p}$ estimate for $p>2$. Last segment is devoted to prove $W^{1,p}$ estimate for $1<p<2$.

We end this section with some notations. We will use  $\fint_E u$ to denote the average of $u$ over the set $E$; i.e.
$$
\fint_E u =\frac{1}{|E|} \int_E u.
$$
Let $B(x,r)$ denote the sphere centered at $x$ with radius $r$. Denote $D(x,r)=B(x,r)\cap\Omega$ and $\Delta(x,r)=B(x,r)\cap\partial\Omega$.
For $R>0$ large sufficiently, let
$$\Omega_R=\{(x',x_d)\in\R^d: |x'|<R\text{ and }\phi(x')<x_d<\phi(x')+R\}$$
where $\phi:\mathbb{R}^{d-1}\to\mathbb{R}$ is the convex function.

\section{\bf Preliminaries}\label{section-2}

We first note that $V$ is a $B_\infty$ weight, defined by Franchi \cite[p.153]{Franchi-1991}. Then the measure $Vdx$ is doubling, i.e., there exists $C>0$ such that for any ball $B$ in $\mathbb{R}^d$,
	\begin{equation}\label{doubling-condition}
		\int_{2B}Vdx\leq C\int_BVdx.
	\end{equation}
Let
$$
\psi(x, r)=\frac{1}{r^{d-2}} \int_{B(x, r)} V(y) dy
$$
for $x \in \mathbb{R}^d, r>0$, then the Fefferman-Phong-Shen maximal function is defined as
\begin{equation}\label{maximal function}
m(x, V)=\inf _{r>0}\left\{r^{-1}: \psi(x, r) \leq 1\right\}.
\end{equation}

Several conclusions from \cite{Shen-1994-IUMJ} and \cite{Shen-1995} will be quoted in this section. These lemmas and definitions are related to the concepts of Fefferman and Phong discussed in \cite{Fefferman-1983}. We list some of them below.

\begin{prop}\label{upper-bound-V}
If $V$ satisfies \eqref{V}, then for a.e. $x \in \mathbb{R}^{d}$,
\begin{equation*}
V(x) \leq C m(x, V)^{2}.
\end{equation*}
\end{prop}

\begin{proof}
  See \cite{Shen-1994-IUMJ}.
\end{proof}

\begin{prop}\label{prop2.6}
Assume $V$ satisfies \eqref{V}. Then there exist $C>0$ and $k_0>0$ such that
\begin{align*}
\psi\left(x, \frac{1}{m(x,V)}\right)=1 \quad\text{and}\quad \psi(x, r)\leq C \left\{rm(x,V)\right\}^{k_0}.
\end{align*}
Moreover, $r\sim \hat{r}$ if and only if $\psi(x, r)\sim 1$.
\end{prop}
\begin{proof}
  See \cite{Shen-1994-IUMJ} and \cite[Lemma 1.2 and Lemma 1.8]{Shen-1995}.
\end{proof}

\begin{lemma}\label{lemma-function m}
There exist $C>0, c>0$ and $k_{0}>0$ depending only on $d$ and the constant in \eqref{V}, such that for $x,y$ in $\mathbb{R}^d$,
\begin{align}
cm(y, V)\leq m(x, V) &\leq C m(y, V) \quad \text { if }|x-y| \leq \frac{C}{m(x, V)} \label{estimate-m1},\\
c(1+|x-y| m(x, V))^{-k_{0}} &\leq  \frac{m(x, V)}{m(y, V)}\leq C(1+|x-y| m(x, V))^{k_0 /\left(k_{0}+1\right)}\label{estimate-m2}.
\end{align}
\end{lemma}

\begin{proof}
  See \cite{Shen-1994-IUMJ}.
\end{proof}

Next we shall introduce the Fefferman-Phong type inequality.

\begin{lemma}\label{remark2.11}
  Let $u\in C_0^1(\R^d)$. Assume $V$ satisfies \eqref{V}. Then
    \begin{equation}\label{Fefferman-Phong-1}
      \int_{\Omega}|u(x)|^2m(x,V)^{2}\,dx\leq C\int_{\Omega}|\nabla u|^2\,dx+C\int_{\Omega}V|u|^2\,dx.
    \end{equation}
\end{lemma}
\begin{proof}
  See \cite[Lemma 1.11]{Shen-1994-IUMJ}. It follows from Proposition \ref{prop2.6} and Lemma \ref{lemma-function m} as well as a covering argument.
\end{proof}

A refine version of Lemma \ref{remark2.11} was obtained in \cite{auscher-2007}.

\begin{lemma}\label{Fefferman-Phong}
     Let $u\in C^1(\overline{\Omega})$. Assume $V$ is an $A_\infty$ weight. Then for $x_0\in \Omega$ and $r>0$,
    \begin{align*}
      \min\left\{r^{-2},\fint_{D(x_0,r)}V\,dy\right\}&\int_{D(x_0,r)}|u|^2\,dx\leq C\left\{\int_{D(x_0,r)}|\nabla u|^2\,dx+\int_{D(x_0,r)}|u|^2 V\,dx\right\}.
    \end{align*}
\end{lemma}

\begin{proof}
  See \cite[Lemma 2.1]{auscher-2007}.
\end{proof}

We end this section with a boundary $L^\infty$ estimate and the estimate for the Neumann function.

\begin{lemma}\label{pointwise estimate}
     Suppose $V(x)>0$ a.e. in $\mathbb{R}^d$.   Suppose $-\Delta u+V u=0$ in $D(x_0,r)$, $\frac{\partial u}{\partial \nu} =0$ on $\partial D(x_0,r)\cap\partial\Omega$ and $(\nabla u)^*\in L^2(\partial D(x_0,r)\cap\partial\Omega)$ for some $x_0\in \overline{\Omega}$ and $r>0$. Then
$$
\sup_{x \in D(x_0,\frac r2)}|u(x)| \leq \frac{C_{k}}{\left\{1+r m\left(x_{0}, V\right)\right\}^{k}}\left( \fint_{D(x_0,r)}|u(x)|^{2} d x\right)^{1 / 2}
$$
for any integer $k>0$.
\end{lemma}

\begin{proof}
  See \cite[Lemma 1.12]{Shen-1994-IUMJ}.
\end{proof}

Let $\Gamma(x,y)$ denote the fundamental solution of the Schr\"odinger operator  $-\Delta + V$ in $\mathbb{R}^d$. Fix $x\in \Omega$, let $v^x(y)$ be the solution of $-\Delta u+V u =0$ in $\Omega$ with Neumann data $\frac{\partial \Gamma(x,y)}{\partial \nu_y}$. Let $N(x,y)=\Gamma(x,y)-v^x(y).$ Then we have the following estimate.
 \begin{lemma}
For any $x,y \in \Omega$,
\begin{equation}\label{N-estimate}
  |N (x,y)|  \leq \frac{C_{k}}{(1+|x-y| m(y, V))^{k}} \cdot \frac{1}{|x-y|^{d-2}},
\end{equation}
where $k \geq 1$ is an arbitrary integer.
\end{lemma}
\begin{proof}
    See \cite[Lemma 1.21]{Shen-1994-IUMJ}.
\end{proof}

\section{\bf A sufficient condition}\label{section-3}

The following theorem is a refined real variable argument which established in \cite[Theorem 3.2]{Shen-2007} (see also \cite[Theorem 4.2.3]{Shen-book}) and can be seen as a duality argument of the Calder\'on-Zygmund decomposition. With this, the $W^{1,p}$ estimates follow from the locally weak reverse H\"older inequality.

\begin{thm}\label{real variable method}
Let $E\subset \mathbb{R}^d$ be a bounded Lipschitz domain and $F\in L^2(E)$. Let $p>2$ and $f\in L^q(E)$ for some $2<q<p$. Suppose that for each ball $B$ with $|B|\leqslant \beta |E|$, 
there exist  $F_B$, $R_B$ on $2B$ such that $|F|\leqslant |F_B|+|R_B|$ on $2B\cap E$,
\begin{equation}\label{rvm2-1}
  \left\{\fint_{2B\cap E}|R_B|^pdx\right\}^{\frac{1}{p}}\leqslant C_1\left\{\left(\fint_{\alpha B\cap E}|F|^2dx\right)^{\frac{1}{2}} +\sup_{B\subset B^{\prime}}\left(\fint_{B^{\prime}\cap E}|f|^2dx\right)^{\frac{1}{2}}\right\}
\end{equation}
and
\begin{equation}\label{rvm2-2}
\fint_{2B\cap E}|F_B|^2dx
\leqslant C_2\sup_{B\subset
B^{\prime}}\fint_{B^{\prime}}|f|^2dx +\sigma \fint_{\alpha
B}|F|^2dx
\end{equation}
where $C_1, C_2>0$ and $0<\beta<1<\alpha$.
 Then, if  $0 \leqslant \sigma<\sigma_0=\sigma_0(C_1,C_2,d,p,q,\alpha, \beta)$, we have
\begin{equation}\label{real result}
\left\{\fint_{E}|F|^qdx\right\}^{\frac{1}{q}}\leqslant
C\left\{\left(\fint_{E}|F|^2dx\right)^{\frac{1}{2}}
+\left(\fint_{E}|f|^qdx\right)^{\frac{1}{q}}\right\},
\end{equation}
where $C>0$ depends only on $C_1,C_2,d,p,q, \alpha, \beta$.
\end{thm}
\begin{proof}
  See \cite[Theorem 2.1]{Geng-2018}.
\end{proof}

With the real variable method at disposal, we give the sufficient condition.

\begin{thm}\label{sufficient condition}
  Let $p>2$. Suppose $V(x)>0$ a.e. in $\mathbb{R}^d$. Assume that for any ball $B(x_0,r_0)$ with the property that either $x_0\in \partial\Omega_R$ or $B(x_0,2r_0)\subset\Omega_R$ for $R$ large, the weak reverse H\"older inequality
  \begin{equation}
    \label{rh inequality}
    \bigg\{\fint_{B(x_0,r_0)\cap\Omega_R}(|\nabla v|+V^{\frac12} |v|)^pdx\bigg\}^{\frac{1}{p}}\leq C_0\bigg\{\fint_{B(x_0,2r_0)\cap\Omega_R}(|\nabla v|+V^{\frac12} |v|)^2dx\bigg\}^{\frac{1}{2}}
  \end{equation}
  holds, whenever $v\in W^{1,2}(B(x_0,2r_0)\cap\Omega_R)$ satisfies $-\Delta v+V v =0$ in $B(x_0,2r_0)\cap\Omega_R$ and $\frac{\partial v}{\partial n}  =0$ on $\Delta(x_0,2r_0)\cap\partial\Omega_R$. 
  Let $u\in W^{1,2}(\Omega)$ be a weak solution of \eqref{NP} with $f\in L^p(\Omega,\R^d)$ and $g=0$. Then $u\in W^{1,p}(\Omega)$ and
  \begin{equation}
    \label{result1}
    \|\nabla u\|_{L^p(\Omega)}+\|V^{\frac12} u\|_{L^p(\Omega)}\leq C\|f\|_{L^p(\Omega)},
  \end{equation}
  with constant $C>0$ depending only on $d,p,C_0$ and the Lipschitz character of $\Omega$.
\end{thm}

\begin{proof}
 For $R>0$ large sufficiently, let
 \begin{align*}
  f_R(x)=\begin{cases}
    f(x),\quad& x\in\Omega_R\text{ or }x\in\partial \Omega\cap\partial \Omega_R,\\[0.1cm]
    0, &\text{otherwise}.
  \end{cases}
\end{align*}
 By taking the limit $R\to \infty$, it suffices for us to show \begin{equation}\label{3.6}
    \|\nabla u\|_{L^p(\Omega_R)}+\|V^{\frac12} u\|_{L^p(\Omega_R)}\leq C\|f_R\|_{L^p(\Omega_R)}
  \end{equation}
  where   \begin{equation}
    -\Delta  u+V u =\mbox{div}\,f_R \quad\text{ in } \Omega,\quad\text{and}\quad\frac{\partial u}{\partial \nu}  =-f_R\cdot n \quad  \text{ on } \partial\Omega.
  \end{equation}
Given any ball $B(x,r)$ satisfying $|B(x,r)| \leqslant \beta|\Omega_R|$ and either $B(x,2r) \subset \Omega_R$ or $B(x,r)$ centers on $\partial \Omega_R$, we set a cut-off function $\varphi \in C_{0}^{\infty}(B(x,8r))$ such that $\varphi=1$ in $B(x,4r)$ and $\varphi=0$ outside $B(x,8r)$. Let $u_1$ be the solution of
\begin{align}\label{Equation v}
-\Delta u_1+Vu_1=\operatorname{div}(\varphi f_R) \quad \text {in } \Omega_R, \quad\text{and}\quad\frac{\partial u_1}{\partial \nu}=-\varphi f_R \cdot n \quad \text {on }\partial \Omega_R.
\end{align}
Let $u_2=u-u_1$ and $D_R(x,t r)=B(x,t r)\cap\Omega_R$, then
\begin{equation}\label{Equation w2}
-\Delta u_2+Vu_2=0\quad \text {in } D_R(x,4 r)\quad \text{ and }\quad \displaystyle\frac{\partial u_2}{\partial \nu}=0\quad  \text{on } \Delta(x,4r)\cap\partial\Omega_R.
\end{equation}
To apply Theorem \ref{real variable method}, let $F=|\nabla u| +V^{\frac12} |u|, F_{B}=|\nabla u_1|+V^{\frac12} |u_1|$ and $R_{B}=|\nabla u_2|+V^{\frac12} |u_2|$, Thus $|F| \leqslant\left|F_{B}\right|+\left|R_{B}\right|$. Then it follows from integration by parts to \eqref{Equation v} that
$$ \fint_{D_R(x,2 r)}\left|F_{B}\right|^{2}\, d x \leqslant \frac{C}{|D_R(x,2 r)|} \int_{\Omega_R}(|\nabla u_1|^{2}+Vu_1^2)\, d x\leqslant C \fint_{D_R(x,8 r)}|f_R|^{2}\, d x.$$
Claim that the weak reverse H\"older inequality
 \begin{equation}\label{rh ineq 1}
   \left\{ \fint_{D_R(x,2 r)}\left|R_{B}\right|^{p} d x\right\}^{\frac{1}{p}}\leqslant C\left\{ \fint_{D_R(x,4 r)}(|\nabla u_2|^{2}+Vu_2^2) d x\right\}^{\frac{1}{2}}
 \end{equation}
holds for a moment, and we obtain
$$
\begin{aligned}
\left\{ \fint_{D_R(x,2 r)}\left|R_{B}\right|^{p} d x\right\}^{\frac{1}{p}}&\leqslant C\left\{ \fint_{D_R(x,4 r)}(|\nabla u|^{2}+Vu^2) d x+ \fint_{D_R(x,4 r)}(|\nabla u_1|^{2}+Vu_1^2) d x\right\}^{\frac{1}{2}} \\
&\leqslant C\left\{ \fint_{D_R(x,4 r)}|F|^{2} d x\right\}^{\frac{1}{2}} +C\left\{ \fint_{D_R(x,8 r)}|f_R|^{2} d x\right\}^{\frac{1}{2}}.
\end{aligned}
$$
Hence by Theorem \ref{real variable method} and the self-improving property of the reverse H\"older condition
\begin{equation}
  \label{rh1}
  \left\{ \fint_{\Omega_R}(|\nabla u|+V^{\frac12} |u|)^{p} d x\right\}^{\frac{1}{p}} \leqslant C\left\{\left( \fint_{\Omega_R}(|\nabla u|+V^{\frac12} u)^{2} d x\right)^{\frac{1}{2}}+\left( \fint_{\Omega_R}|f_R|^{p} d x\right)^{\frac{1}{p}}\right\}.
\end{equation}
This, combining with integration by parts as well as H\"older's inequality, gives \eqref{3.6}.
\end{proof}


To establish the reverse H\"older inequality, we need an auxiliary lemma as follows.

\begin{lemma}\label{convex-estimate}
  Suppose $V>0$ and $\Omega$ is the region above a convex graph in $\R^d$ with $C^2$ boundary. Assume $u$ is a weak solution of $-\Delta u+Vu=0$ in $D(x_0,2r)$ and $\frac{\partial v}{\partial n}  =0$ on $\Delta(x_0,2r)$. Then for $p>1$ and $\frac{1}{q}=\frac{1}{p}-\frac{1}{d}$,
\begin{equation}
\label{6.14}
\left\{\int_{B(x_0,r)\cap\Omega_R} |\nabla u|^{ q} dx\right\}^{\frac{1}{q}}\leqslant Cr^{-1}\left\{\int_{B(x_0,2r)\cap\Omega_R} \left(|\nabla u| +rV|u|\right)^{p} d x\right\}^{\frac{1}{p}}
\end{equation}
 where  $\varphi\in C^\infty_0(B(x_0,2r)\cap\Omega_R)$.
\end{lemma}

\begin{proof}
Fix $0<\rho<\tau<\infty$, for $g\in G:=\{g=(g_1,\cdots,g_d)\in (C^2_0(\overline{\Omega}))^d:g\cdot n=0\text{ on }\partial\Omega\}$ let $h_g:\overline{\Omega}\to [0,1]$ be continuous so that
\begin{align*}
  h_g(x)=
  \begin{cases}
  0,\quad &x\in I_{g}:=\{x\in\Omega: |g(x)|^{2}\leq\rho\},\\
  \frac{1}{\tau-\rho}(|g(x)|^2-\rho),\quad & x\in II_{g}:=\{x\in\Omega: \rho<|g(x)|^{2}<\tau\},\\
  1,\quad &x\in III_{g}:=\{x\in\Omega: |g(x)|^{2}\geq\tau\}.
  \end{cases}
\end{align*}
 It follows from integration by parts that
\begin{equation*}\label{5.8}
  \begin{aligned}
&2 \int_{\Omega}  h_g^{\prime} g_kg_i\frac{\partial g_k}{\partial x_j}\frac{\partial g_j}{\partial x_i}\, d x-\int_{\partial \Omega} h_g\left\{g_in_j\frac{\partial g_j}{\partial x_i}-g_in_i {\rm div}g\right\} d \sigma\\
&\qquad=\int_{\Omega} h_g\left\{({\rm div}g)^{2}-\frac{\partial g_i}{\partial x_j}\frac{\partial g_j}{\partial x_i}\right\} d x+2 \int_{\Omega} h_g^{\prime}g_kg_i\frac{\partial g_k}{\partial x_i}{\rm div}g\, d x
\end{aligned}
\end{equation*}
where $\sigma= H^{d-1}$ denotes the $(d-1)$-dimensional Hausdorff measure. Let $\beta(\cdot,\cdot)$ denote the second fundamental quadratic form of $\partial\Omega$ (see \cite[pp.133-134]{Grisvard-book-2011}). The convexity
  $g_in_i{\rm div}g-g_in_j\frac{\partial g_j}{\partial x_i}=-\beta(g-(g\cdot n)n,g-(g\cdot n)n)\geq0\text{ on } \partial\Omega$,
gives that
\begin{equation}
  \label{3.1}
 \begin{aligned}
   \int_{II_g}|v|^{2 } d x &\leqslant 2 \int_{II_g}|v||g|\bigg\{\bigg(\sum_{i,j} \left|\frac{\partial g_i}{\partial x_j}-\frac{\partial g_j}{\partial x_i}\right|^2\bigg)^{\frac12}+|{\rm div}g|\bigg\}dx\\
   &\qquad\qquad+2(\tau-\rho) \int_{II_g\cup III_g}h_g \left\{|{\rm div}g|^2- \frac{\partial g_i}{\partial x_j}\frac{\partial g_j}{\partial x_i}\right\}dx.
 \end{aligned}
\end{equation}
where $v=\nabla|g|^2$ and Cauchy's inequality was also used. 
Take $g=(\nabla u)\varphi$ in \eqref{3.1} where $\varphi\in C^\infty_0(B(x_0,2r)\cap\Omega_R)$ such that $\varphi=1$ in $B(x_0,r)\cap\Omega_R$ and $|\nabla \varphi|\leq Cr^{-1}$. It is easy to verify
$$\frac{\partial g_i}{\partial x_{j}}=\varphi \frac{\partial^{2} u}{\partial x_{i} \partial x_{j}}+\frac{\partial u}{\partial x_{i}} \frac{\partial \varphi}{\partial x_{j}}$$
and
$${\rm div}g={\rm div}((\nabla u)\varphi)=(\Delta u) \varphi+\nabla u \cdot \nabla \varphi=\nabla u \cdot \nabla \varphi+Vu\varphi.$$
Note that
\begin{equation*}
  \label{6.10}
  \begin{aligned}
\bigg(\sum_{i,j} \left|\frac{\partial g_i}{\partial x_j}-\frac{\partial g_j}{\partial x_i}\right|^2\bigg)^{\frac12}+|{\rm div}g|
& \leqslant\left\{2|\nabla u|^{2}|\nabla \varphi|^{2}+2|\nabla u \cdot \nabla \varphi|^{2}\right\}^{\frac12}+|\nabla u \cdot \nabla \varphi|+V|u||\varphi|\\
& \leqslant C|\nabla u||\nabla \varphi|+V|u||\varphi|
\end{aligned}
\end{equation*}
and
\begin{equation*}
  \label{6.11}
  \begin{aligned}
  |{\rm div}g|^{2}-\frac{\partial g_i}{\partial x_j}\frac{\partial g_j}{\partial x_i}
  & \leqslant 2V|\nabla u \cdot \nabla\varphi||u||\varphi| +V^2|u|^2|\varphi|^2 -\varphi^{2}\left|\nabla^{2} u\right|^{2}-2 \varphi \frac{\partial^{2} u}{\partial x_{i} \partial x_{j}} \frac{\partial u}{\partial x_{i}} \frac{\partial \varphi}{\partial x_{j}} \\
  & \leqslant-\sum_{i, j}\left(\varphi \frac{\partial^{2} u}{\partial x_{i} \partial x_{j}}+\frac{\partial u}{\partial x_{i}} \frac{\partial \varphi}{\partial x_{j}}\right)^{2}+2|\nabla u|^{2}|\nabla \varphi|^{2}+V^2|u|^2|\varphi|^2 \\
  & \leqslant 2|\nabla u|^{2}|\nabla \varphi|^{2}+V^2|u|^2|\varphi|^2.
  \end{aligned}
\end{equation*}
By using the co-area formula repeatedly, we have
$$
\begin{aligned}
\int_{\rho }^{\tau} \int_{\{|g|^2=s\}}|v| d \sigma d s \leqslant & C\int_{\rho }^{\tau} \int_{\{|g|^2=s\}}|g|h\, d \sigma d s +C(\tau-\rho )\int_{\{|g|^2>\rho \}} h_gh^2 d x
\end{aligned}
$$
where $h=|\nabla u||\nabla \varphi| +V|u||\varphi|$. Taking $\tau \rightarrow \rho^{+}$, we obtain that for $\rho\in(0,\infty)$,
\begin{equation}
  \label{6.9}
\begin{aligned}
\int_{\{|g|^2=\rho\}}|v|\, d \sigma \leqslant & C\rho^{\frac12} \int_{\{|g|^2=\rho\}}h\, d \sigma +C\int_{\{|g|^2>\rho \}}h^2 d x.
\end{aligned}
\end{equation}
where Lebesgue's differentiation theorem is also used.

Without loss of generality, assume that $|(\nabla u)\varphi|^2$ is bounded from below by a positive constant. Multiplying both sides of \eqref{6.9} by $\rho^{b-2}$ and integrating the resulting inequality in $\rho$ over $(0,\infty)$, we obtain that for $b>1$,
$$
\begin{aligned}
\int_{\Omega} |(\nabla u)\varphi|^{2b-4}|v|^{2} d x =\int_{0}^{\infty} \rho^{a} \int_{\{|g|^2=\rho\}}|v| d \sigma d \rho \\
\leqslant C\varepsilon\int_{\Omega} |(\nabla u)\varphi|^{2b-4}|v|^2d x+ C \int_{\Omega} |(\nabla u)\varphi|^{2b-2} h^{2} d x
\end{aligned}
$$
where the co-area formula and the Cauchy's inequality are used. Then by Poincar\'e inequality,
\begin{equation}
  \label{6.13}
  \left\{\int_{\Omega} |(\nabla u)\varphi|^{ b 2^*} dx\right\}^{\frac{2}{2^*}}\leqslant C\int_{\Omega} |(\nabla u)\varphi|^{2b-4}|v|^{2} d x \leqslant C \int_{\Omega} |(\nabla u)\varphi|^{2b-2} h^{2} d x
\end{equation}
where $2^*=\frac{2 d}{d-2}$. Using H\"older's inequality, we obtain for $p',p>1$,
\begin{equation}
  \label{6.15}
   \int_{\Omega} |(\nabla u)\varphi|^{2b-2} h^{2} d x \leqslant\left\{\int_{\Omega} |(\nabla u)\varphi|^{(2b-2) \frac{p'}2 } d x\right\}^{\frac2 { p'}}\left\{\int_{\Omega} h^{ p} d x\right\}^{\frac2 { p}},
\end{equation}
where $\frac1 { p'}+\frac1 {p}=\frac12$. Choose $p'$ so that $(b-1) p'=b 2^*$ and let $q=b 2^*$.  A direct computation leads $\frac{1}{q}=\frac{1}{p}-\frac{1}{d}$ and \eqref{6.14}. This completes the proof.
\end{proof}

\begin{thm}\label{rh inequality thm}
 Assume $V>0$ satisfies \eqref{V} and $\Omega$ is the region above a convex graph in $\R^d$ with $C^2$ boundary. Then the weak reverse H\"older inequality \eqref{rh inequality} holds for any $p>2$.
\end{thm}

\begin{proof}
 Denote $D_R(x,r)=B(x,r)\cap\Omega_R$. With Lemma \ref{convex-estimate} at disposal, we obtain that for all $p>1$ and $\frac1q=\frac1p-\frac1d$,
 \begin{align*}
   \left\{\fint_{D_R(x_0,r)}|\nabla u|^{ q} d x\right\}^{\frac{1}{q}} 
   &\leqslant  C\left\{\fint_{D_R(x_0,2r)}|\nabla u|^{p} d x\right\}^{\frac{1}{p}}
   +Cr\left\{\fint_{D_R(x_0,2r)}|V u|^{p} d x\right\}^{\frac{1}{p}}.
 \end{align*}
 Using Lemma \ref{pointwise estimate} and \eqref{V}, we have
 \begin{align*}
   r\left\{\fint_{D_R(x_0,r)}|V u|^{ p} d x\right\}^{\frac{1}{p}} &\leqslant  Cr\left(\fint_{D_R(x_0,r)}V^p\, d x\right)^{\frac1p}\sup_{D_R(x_0,r)}|u|\\
   &\leqslant  \frac{Cr^{1-\frac{d}{2}}}{\left\{1+r m\left(x_{0}, V\right)\right\}^{k}}\fint_{D_R(x_0,2r)}V\,dx \left( \int_{D_R(x_0,2r)}|u(x)|^{2} d x\right)^{\frac12}
 \end{align*}
 If $r^2\fint_{D(x_0,r)}V\,dx\leq1$, it follows from Lemma \ref{Fefferman-Phong} and H\"older's inequality that for $p\geq2$,
 \begin{align*}
   r\left\{\fint_{D_R(x_0,r)}|V u|^{ p} d x\right\}^{\frac{1}{p}}
   &\leqslant  Cr^{-\frac{d}{2}}\left(r^2\fint_{D_R(x_0,r)}V\,dx\right)^{\frac12}  \left( \fint_{D_R(x_0,2r)}V\,dx\int_{D_R(x_0,2r)}|u(x)|^{2} d x\right)^{\frac12}\\
   &\leqslant  C\left\{\fint_{D_R(x_0,2r)}(|\nabla u| +|V^{\frac12} u|)^{p} d x\right\}^{\frac{1}{p}}.
 \end{align*}
 In the case of $r^2\fint_{D(x_0,r)}V\,dx>1$, it follows from Proposition \ref{prop2.6} and Lemma \ref{Fefferman-Phong} that
 \begin{align*}
   r\left\{\fint_{D_R(x_0,r)}|V u|^{ p} d x\right\}^{\frac{1}{p}}
   &\leqslant  \frac{Cr^{-\frac{d}{2}}\cdot r^2\fint_{D_R(x_0,r)}V\,dx}{\left\{1+r m\left(x_{0}, V\right)\right\}^{k}} \left( r^{-2}\int_{D_R(x_0,2r)}|u(x)|^{2} d x\right)^{\frac12}\\
   &\leqslant  \frac{C\left\{ r m(x_{0}, V)\right\}^{k_0}}{\left\{1+r m\left(x_{0}, V\right)\right\}^{k}} \left\{\fint_{D_R(x_0,2r)}(|\nabla u| +|V^{\frac12} u|)^{p} d x\right\}^{\frac{1}{p}}\\
   &\leqslant  C\left\{\fint_{D_R(x_0,2r)}(|\nabla u| +|V^{\frac12} u|)^{p} d x\right\}^{\frac{1}{p}}
 \end{align*}
 if we choose $k=k_0$. This gives
 \begin{align*}
   r\left\{\fint_{D_R(x_0,r)}|V u|^{ p} d x\right\}^{\frac{1}{p}}
   &\leqslant  C\left\{\fint_{D_R(x_0,2r)}(|\nabla u| +|V^{\frac12} u|)^{p} d x\right\}^{\frac{1}{p}}
 \end{align*}
 and in similar manner,
 \begin{align*}
   \left\{\fint_{D_R(x_0,r)}|V^{\frac12} u|^{ q} d x\right\}^{\frac{1}{q}}
   &\leqslant  C\left\{\fint_{D_R(x_0,2r)}(|\nabla u| +|V^{\frac12} u|)^{p} d x\right\}^{\frac{1}{p}}.
 \end{align*}
 By a iteration and the self-improvement, we have for $p>2$
 \begin{equation}\label{5.7}
   \left\{\fint_{D_R(x_0,r)}(|\nabla u| +|V^{\frac12} u|)^{p} d x\right\}^{\frac{1}{p}}
   \leqslant  C\left\{\fint_{D_R(x_0,2 r)}(|\nabla u| +|V^{\frac12} u|)^{ 2} d x\right\}^{\frac{1}{2}}.
 \end{equation}
\end{proof}

\section{Duality argument}

\begin{lemma}\label{term1-1<p<2}
  Let $\Omega$ be the region above a convex graph in $\R^d$ with $C^2$ boundary. Suppose $V$ satisfies \eqref{V}. Assume $$1<p<\infty.$$ Let $u\in W^{1,2}(\Omega)$ be a weak solution of \eqref{NP} with $f\in L^p(\Omega,\R^d)$ and $g=0$. Then $u\in W^{1,p}(\Omega)$ and
  \begin{equation}
    \label{result1-1}
    \|\nabla u\|_{L^p(\Omega)}\leq C\|f\|_{L^p(\Omega)},
  \end{equation}
  with constant $C$ depending only on $d,p$ and the Lipschitz character of $\Omega$.
\end{lemma}

\begin{proof}
  Theorem \ref{rh inequality thm}, together with Theorem \ref{sufficient condition}, gives that $u \in W^{1, p}(\Omega)$ and that for any $q>2$,
  \begin{equation}\label{term1}
    \|\nabla u\|_{L^{q}(\Omega)} \leqslant C\|f\|_{L^{q}(\Omega)}.
  \end{equation}
  Let $h \in C_{0}^{\infty}(\Omega,\R^d)$ and $v$ be a weak solution of $-\Delta v+Vv=\operatorname{div} h$ in $\Omega$ and $\frac{\partial v}{\partial \nu}=0$ on $\partial \Omega$. Suppose $p, q$ are conjugate.
The weak formulations of variational solution of $u$ and $v$ imply that
\begin{equation}
  \label{24}
 \left|\int_{\Omega} h_{i} \frac{\partial u}{\partial x_{i}} d x\right|=\left|\int_{\Omega} f_{i} \frac{\partial v}{\partial x_{i}} d x\right|\leqslant\|f\|_{L^{p}(\Omega)}\|\nabla v\|_{L^{q}(\Omega)} \leqslant C\|f\|_{L^{p}(\Omega)}\|h\|_{L^{q}(\Omega)}.
\end{equation}
where H\"older's inequality and \eqref{term1} are also used. This gives that for $1<p<2$
\begin{equation}
  \label{26}
  \|\nabla u\|_{L^{p}(\Omega)}=\sup _{\|h\|_{L^{q}(\Omega)} \leqslant 1}|\langle h, \nabla u\rangle| \leqslant C\|f\|_{L^{p}(\Omega)},
\end{equation}
and thus \eqref{result1-1} holds for all $1<p<\infty$ in the region above a convex graph.
\end{proof}

\begin{lemma}\label{term1-g}
  Assume $\Omega$ and $V$ are same as in Lemma \ref{term1-1<p<2}. Let $$1<p<\infty.$$ Then the solution $u\in W^{1,p}(\Omega)$ to \eqref{NP} with $g\in B^{-\frac 1p,p}(\partial\Omega)$ and $f=0$  satisfies
  \begin{equation}
    \label{result1-2}
    \|\nabla u\|_{L^p(\Omega)}\leq C\|g\|_{B^{-\frac 1p,p}(\partial\Omega)},
  \end{equation}
  where $C$ depends only on $d,p$ and the Lipschitz character of $\Omega$.
\end{lemma}

\begin{proof}
  Let $h\in C^\infty_0(\Omega)$ and $w$ be the weak solution to
$$-\Delta (v-c)+V(v-c) ={\rm div}\;h\quad  \text{in } \Omega,\quad\text{ and }\quad\frac{\partial v}{\partial \nu}  =0 \quad\text{on } \partial\Omega,$$
where $c=\fint_{\Omega}v\,dx$. Then the weak formulation, the Sobolev embedding and Poincar\'e inequality imply that
\begin{equation}\label{3.3}
  \begin{aligned}
    \left|\int_\Omega h\cdot\nabla u\, dx\right|=\left|\int_{\partial\Omega} g(v-c)\, dx\right| &\leq\|g\|_{B^{-\frac 1p,p}(\partial\Omega)} \|v-c\|_{B^{\frac 1p,q}(\partial\Omega)}\\
    &\leq\|g\|_{B^{-\frac 1p,p}(\partial\Omega)} \|v-c\|_{W^{1,q}(\Omega)}\\
    &\leq\|g\|_{B^{-\frac 1p,p}(\partial\Omega)} \|\nabla v\|_{L^{q}(\Omega)}
  \end{aligned}
\end{equation}
where $p, q$ are conjugate. It follows from Lemma \ref{term1-1<p<2} that for $1<q<\infty$, $$\|\nabla v\|_{L^{q}(\Omega)}=\|\nabla (v-c)\|_{L^{q}(\Omega)}\leq C\|h\|_{L^{q}(\Omega)}.$$
This gives
\begin{equation}\label{6.2}
  \|\nabla u\|_{L^p(\Omega)}=\sup_{\|h\|_{L^{q}(\Omega)}\leq 1}\left|\int_\Omega h\cdot\nabla u\, dx\right|\leq C\|g\|_{B^{-\frac 1p,p}(\partial\Omega)}
\end{equation}
and thus \eqref{result1-2} holds for $1<p<\infty$.
\end{proof}

Finally we are in a position to give the proof of Theorem \ref{thm1}.

\begin{proof}[Proof of Theorem \ref{thm1}]
It follows directly from Lemma \ref{term1-1<p<2} and Lemma \ref{term1-g} that
\begin{equation*}
  \|\nabla u\|_{L^p(\Omega)}\leq C\left\{\|f\|_{L^p(\Omega)}+\|g\|_{B^{-\frac 1p,p}(\partial\Omega)}\right\}
\end{equation*}
for $1<p<\infty$. Next, to show
\begin{equation}\label{Vu-estimate}
  \|V^\frac{1}{2}u\|_{L^p(\Omega)}\leq C\left\{\|f\|_{L^p(\Omega)}+\|g\|_{B^{-\frac 1p,p}(\partial\Omega)}\right\},
\end{equation}
  decompose $u=u_1+u_2$ where $u_1,u_2$ are weak solutions of
\begin{equation*}
  \begin{aligned}
  \begin{cases}
-\Delta u_1+Vu_1 ={\rm div} f  &\text{in } \Omega,\\[0.1cm]
\hspace{3.35em}\frac{\partial u_1}{\partial \nu}  =-f\cdot n&\text{on } \partial\Omega,
  \end{cases}
\end{aligned}
\quad\text{ and }\quad
\begin{aligned}
\begin{cases}
-\Delta u_2+Vu_2 =0\quad  &\text{in } \Omega,\\[0.1cm]
\hspace{3.35em}\frac{\partial u_2}{\partial \nu}  =g\ &\text{on } \partial\Omega.
  \end{cases}
\end{aligned}
\end{equation*}
It follows from the Poisson representation formula and integration by parts that
  $$u_1(x)= \int_{\partial\Omega}  N(x, y)\frac{\partial u_1}{\partial \nu}d \sigma(y)+ \int_{\Omega}  N(x, y)(-\Delta +V)u_1\,d y =  -\int_{\Omega} \nabla_y N(x, y)f(y)d y.$$
By H\"older's inequality we have
\begin{equation}\label{4.6}
|u_1(x)| \leq \left\{\int_{\Omega}|\nabla_y N(x, y)| d y\right\}^{\frac1 q} \left\{\int_{\Omega}|\nabla_y N(x, y)||f(y)|^{p} d y\right\}^{\frac1 p}
\end{equation}
where $q=\frac{p}{p-1}$. Fix $x\in \partial\Omega$. Let $r_0=\frac{1}{m(x,V)}$ and $E_j=\{y\in\Omega: |x-y|\sim2^jr_0\}$. It follows from \eqref{N-estimate} and Caccippoli's inequality that
\begin{equation}\label{N-estimate-1}
  \begin{aligned}
    \int_{E_j}|\nabla_y N(x, y)|\, d y & \leq C(2^jr_0)^{\frac d2}\left(\int_{E_j}|\nabla_y N(x, y)|^2\, d y\right)^{\frac12}   \leq C(2^jr_0)^{\frac d2-1}\left(\int_{E_j}|N(x, y)|^2\, d y\right)^{\frac12}\\
    & \leq C(2^jr_0)^{\frac d2-1}\cdot\frac{(2^jr_0)^{\frac d2}}{(1+2^j)^{k}(2^jr_0)^{d-2}}=\frac{C 2^jr_0}{(1+2^j)^{k}}
  \end{aligned}
\end{equation}
where H\"older's inequality was also used in the first inequality. Taking $k=2$, we have
$$
\begin{aligned}
|u_1(x)| & \leq \frac{C}{m(x, V)^{1 / q}}\left\{\int_{\Omega}|\nabla_y N(x, y) \| f(y)|^{p} d y\right\}^{1 / p}
\end{aligned}
$$
This combining with Proposition \ref{upper-bound-V} gives that
\begin{align*}
  \int_{\Omega}|V^{\frac12}(x)u_1(x)|^{p}d x&\leq C \int_{\Omega}\left|m(x, V)u_1\right|^{p} d x\leq C \int_{\Omega}|f(y)|^{p}\left\{\int_{\Omega} m(x, V)|\nabla_y N(x, y)| d x\right\} d y.
\end{align*}
For fixed $y\in\partial\Omega$, Let $r_1=\frac{1}{m(y,V)}$ and $F_j=\{x\in\Omega: |x-y|\sim2^jr_1\}$. Together Lemma \ref{lemma-function m} with \eqref{N-estimate-1} yields that
$$
\int_{F_j}|\nabla_y N(x, y)|m(x, V)\, d x  \leq  \frac{C 2^j r_1}{(1+2^j)^{k}}\cdot(1+2^j)^{k_0}r_1^{-1}= \frac{C 2^j}{(1+2^j)^{2}}
$$
where $k$ is chosen to be $k_0+2$ in the second inequality. Thus we have
\begin{equation}\label{3.2}
\int_{\Omega} m(x, V)|\nabla_y N(x, y)| d x\leq C\sum_{j=-\infty}^{\infty}\frac{2^j}{(1+2^j)^{2}}\leq C
\end{equation}
which implies for $1<p<\infty$,
\begin{equation}\label{term2}
  \|V^\frac{1}{2}u_1\|_{L^p(\Omega)}\leq C\|f\|_{L^p(\Omega)}.
\end{equation}
Let $h\in C_0^\infty(\Omega)$ and $v$ solves
\begin{align*}
  \begin{cases}
-\Delta v+Vv =h  &\text{in } \Omega,\\[0.1cm]
\hspace{2.8em}\frac{\partial v}{\partial \nu}  =0&\text{on } \partial\Omega.
  \end{cases}
\end{align*}
Then as in \eqref{3.3}
\begin{align*}
  \left|\int_\Omega u_2 h \, dx\right|=\left|\int_{\partial\Omega} gv\, d\sigma\right|\leq \|g\|_{B^{-\frac 1p,p}(\partial\Omega)} \|\nabla v\|_{L^{q}(\Omega)}.
\end{align*}
By a duality argument, it suffices to show that
\begin{equation}\label{6.7}
  \int_{\Omega}|\nabla v|^q\, dx\leq C\int_{\Omega}\frac{|h(x)|^q}{m(x,V)^q}\, dx .
\end{equation}
To show \eqref{6.7}, note that
\begin{equation}\label{3.5}
  \begin{aligned}
    |\nabla v(x)|&=\left|\int_{\Omega} \nabla_x N(x, y)h(y)\,d\sigma(y)\right|\\
    &\leq C\left(\int_{\Omega} |\nabla_x N(x, y)|m(y,V)^p\,dy\right)^{\frac1p} \left(\int_{\Omega} |\nabla_x N(x, y)|\frac{|h(y)|^q}{m(y,V)^q}\,dy\right)^{\frac1q}.
  \end{aligned}
\end{equation}
A similar computation as \eqref{3.2} shows
\begin{equation}\label{3.4}
\int_{\Omega} |\nabla_x N(x, y)|m(y,V)^p\,dy\leq Cm(x,V)^{p-1}.
\end{equation}
Plugging \eqref{3.2} and \eqref{3.4} into \eqref{3.5} gives that
  \begin{align*}
    \int_{\Omega}|\nabla v|^q\, dx&\leq C\int_{\Omega} \frac{|h(y)|^q}{m(y,V)^q} \int_{\Omega} m(x,V)|\nabla_x N(x, y)|\,dx dy\leq C\int_{\Omega} \frac{|h(y)|^q}{m(y,V)^q}dy.
  \end{align*}

The uniqueness for $p>2$ and $1<p<2$ follows from the uniqueness for $p=2$ and the duality argument. And a limit argument leads the existence.
\end{proof}

\newpage
\bibliographystyle{amsplain}
\bibliography{Schrodinger}

\bigskip

\begin{flushleft}
Ziyi Xu,
School of Mathematics and Statistics,
Lanzhou University,
Lanzhou, P.R. China.

E-mail: 120220907780@lzu.edu.cn
\end{flushleft}
\bigskip
\end{document}